\documentclass[reqno]{amsart}

\usepackage[latin1]{inputenc}
\usepackage{amsmath}
\usepackage{amsfonts}
\usepackage{amssymb}
\usepackage{graphics}
\usepackage{color}

\usepackage{enumerate}
\usepackage{amssymb,amsmath,amsthm,amscd}
\usepackage{latexsym,verbatim,graphicx,amsfonts}

\usepackage{amscd}
\usepackage{amsmath}
\usepackage{amssymb}
\usepackage{amsthm}
\usepackage{latexsym}
\usepackage{verbatim}
\usepackage{hyperref}

\theoremstyle{plain}
\newtheorem{theorem}{Theorem}[section]
\newtheorem{thm}[theorem]{Theorem}
\newtheorem{thm*}{Theorem}

\newtheorem{cor}[theorem]{Corollary}

\newtheorem{prop}[theorem]{Proposition}
\theoremstyle{definition}

\newtheorem{ex}[theorem]{Example}
\theoremstyle{remark}
\newtheorem{rem}[theorem]{Remark}

\newcommand{\R}{\mathbb{R}}
\newcommand{\C}{\mathbb{C}}

\newcommand{\DD}{\mathbb{D}}
\newcommand{\N}{\mathbb{N}}

\def\MD{\mathcal{M}(\mathcal{D})}
\def\D{\mathcal D}

\subjclass[2010]{Primary 47A15; Secondary 30H10, 30H20.}

\begin{document}
\title[Wandering property]{On the wandering property in Dirichlet spaces}
\author[Gallardo-Guti\'errez]{Eva A. Gallardo-Guti\'errez}
\address{Eva A. Gallardo-Guti\'errez \newline Departamento de An\'alisis Matem\'atico y Matem\'atica Aplicada,\newline
Facultad de Matem\'aticas,
\newline Universidad Complutense de
Madrid, \newline
 Plaza de Ciencias N$^{\underbar{\Tiny o}}$ 3, 28040 Madrid,  Spain
 \newline
and Instituto de Ciencias Matem\'aticas (CSIC-UAM-UC3M-UCM),
\newline Madrid,  Spain }\email{eva.gallardo@mat.ucm.es}

\author[Partington]{Jonathan R. Partington}
\address{Jonathan R. Partington \newline School of Mathematics, \newline  University of Leeds,\newline  Leeds LS2 9JT, U. K.} \email{J.R.Partington@leeds.ac.uk}

\author[Seco]{Daniel Seco}
\address{Daniel Seco \newline Universidad Carlos III de Madrid and Instituto de Ciencias Matem\'aticas (CSIC-UAM-UC3M-UCM), Departamento de Matem\'aticas UC3M, Avenida de la Universidad 30,  28911 Legan\'es (Madrid),
Spain.} \email{dsf$\underline{\,\,\,}$cm@yahoo.es}

\thanks{This work is partially supported by Plan Nacional  I+D grant no. MTM2016-77710-P, Spain. E. A. Gallardo-Guti\'errez and D. Seco are also supported by the European Regional Development
Fund and by Severo Ochoa Programme for Centers of Excellence in R$\&$D, project SEV-2015-0554 (Spain), which, in particular gave support in the form of a postdoctoral grant for Seco at Instituto de Ciencias Matem\'aticas ICMAT}


\date{\today}


\begin{abstract}
We show that in a scale of weighted Dirichlet spaces $D_{\alpha}$, including the Bergman space, given any finite Blaschke  product $B$ there exists  an equivalent norm in $D_{\alpha}$ such that $B$ satisfies the \emph{wandering subspace property} with respect to such norm. This extends, in some sense, previous results by Carswell, Duren and Stessin \cite{CDS}. As a particular instance, when $B(z)=z^k$ and $|\alpha| \leq \frac{\log (2)}{\log(k+1)}$, the chosen norm is the usual one in  $D_\alpha$.
\end{abstract}

\maketitle

\section{Introduction}\label{Intro}

For isometries $T$ acting on complex, separable, infinite dimensional Hilbert spaces $\mathcal{H}$, the classical Wold Decomposition Theorem asserts that whenever $T$ is pure ($\bigcap_{n=0}^{\infty} T^n \mathcal{H}=\{0\}$), the closed subspace $\mathcal{K}=\mathcal{H}\ominus T\mathcal{H}$ has the \emph{wandering subspace property} in $\mathcal{H}$: $\mathcal{H}$ coincides with the smallest closed invariant subspace under $T$ generated by $\mathcal{K}$, denoted by $[\mathcal{K}]_T$. This is a consequence of the fact that $\mathcal{H}$ decomposes as the orthogonal direct sum of closed subspaces
$$
\mathcal{H}=\mathcal{K}\oplus T\mathcal{K}\oplus T^2 \mathcal{K}\oplus \dots
$$
 More generally, a subspace of a Hilbert space is called a \emph{wandering subspace} of a given operator if it is orthogonal to its images under positive powers of the operator. In this regards, the Wold Decomposition Theorem says that every invariant subspace of a pure isometry is indeed, generated by a wandering subspace.\smallskip

Well known examples arise when considering multiplication operators induced by inner functions in the classical Hardy space $H^2$. Recall that an inner function $\theta$ is an analytic function in the unit disc $\mathbb{D}$ with contractive values ($|\theta(z)|\leq 1$ for $z\in \mathbb{D}$) such that the boundary  values
$$
\theta(e^{it}):=\lim_{r\to 1^{-}} \theta(re^{it})
$$
have modulus 1 for almost all $t$ (they exist for almost every $t$ with respect to Lebesgue measure on the unit circle). In such cases,  every closed subspace $\mathcal{M}$  in $H^2$ invariant under multiplication by $\theta$  is wandering and
\begin{equation*}
[\mathcal{M}\ominus \theta \mathcal{M}]_{\theta}=\mathcal{M}.
\end{equation*}
Accordingly, $\theta$ is said to have the \emph{wandering subspace property} (WSP).\smallskip

Nevertheless, it is not completely understood yet which functions $\varphi$ in $H^{\infty}$ (the space of bounded analytic functions on $\mathbb{D}$) enjoy the WSP in $H^2$, that is, for which functions do the corresponding multiplication operators $M_{\varphi}$ on $H^2$ satisfy
\[[\mathcal{M}\ominus \varphi \mathcal{M}]_{M_\varphi}=\mathcal{M}\]
for every closed invariant subspace $\mathcal{M}$. In \cite{KhLaSt}, it was shown that a necessary condition is that $\varphi$ be writable as the composition $G\circ h$, where $h$ is an inner function and $G$ is univalent in $\mathbb{D}$. Moreover, they also proved a sufficient condition, namely, $\varphi=G\circ h$ with $G$ a weak-star generator of $H^\infty$.  Whether this last condition is in fact a necessary one is left open.\smallskip

The question turns out to be drastically difficult to handle whenever the underlying Hilbert space is the Bergman space $A^2$. In a remarkable paper, Aleman, Richter and Sundberg \cite{ARS} proved that $\varphi(z)=z$  possesses the WSP in $A^2$. However, for univalent functions, Carswell \cite{Ca} showed the existence of bounded univalent functions $\varphi$  in $\mathbb{D}$, vanishing at the origin and failing to have the WSP both in $H^2$ and $A^2$. Indeed, previously in \cite{CDS}, the authors had provided necessary conditions for $H^{\infty}$ functions to have the WSP in $A^2$. They showed that in particular, not every inner function has this property in $A^2$ and even that infinite Blaschke products without it can be found. For finite Blaschke products, the question in the Bergman space remains open (see \cite{CaW}).\smallskip

The main goal of this work is showing that not only in the Bergman space but also in a scale of weighted Dirichlet spaces $D_{\alpha}$ including $A^2$, for every finite Blaschke product $B$, it is possible to renorm the space (with an equivalent norm) such that $B$ enjoys the wandering subspace property. This seems to follow an opposite direction to a recent work by our third author \cite{Secozk}, in which renormings were found of the same spaces allowing one to \emph{disprove} the corresponding WSP for multiplication by some monomials. In sum, one conclusion of the present work is that the geometry of the space plays a significant role in order to deal with this question, since its answer depends strongly on the norm expression.\smallskip

The rest of the manuscript is organized as follows. In Section \ref{Sect 2} we recall some preliminaries, introducing the family of weighted Dirichlet spaces $D_{\alpha}$, where our work takes place. We will recall Shimorin's Theorem \cite{Shimorin}, which provides a unified proof of the theorems of Beurling \cite{Beu} and Aleman, Richter and Sundberg \cite{ARS} and shows, in particular, that $\varphi(z)=z$  possesses the WSP in the scale of $D_{\alpha}$ spaces considered.
In addition, we introduce some basic results justifying the direction of the proof of our main results, which will be proved in Section \ref{Sect 3}.
Moreover, some consequences are derived, including the observation that for a range of $\alpha$ the WSP holds for $\varphi(z)=z^k$ ($k \geq 1$) even with the original norm. Finally, and even though we were not able to answer this question for $A^2$ with its usual norm, we were able to establish the WSP for $z^k$ acting on its finite codimensional subspace $z^k A^2$.

\section{The setting}\label{Sect 2}

\subsection{Dirichlet-type spaces} Let $\alpha$ be a real number. The Dirichlet-type space $D_{\alpha}$ consists of analytic functions
$f(z)=\sum_{k=0}^{\infty} a_k z^k$ in $\DD$ such that its norm
$$
\|f\|_\alpha := \left ( \sum_{k=0}^{\infty}
|a_k|^2 (k+1)^{\alpha} \right )^{1/2}
$$
is finite. Observe that particular instances of $\alpha's$ yield well-known Hilbert spaces of analytic functions in $\mathbb{D}$. More precisely, when $\alpha = -1$ we have the classical Bergman space $A^2$, $\alpha = 0$ corresponds to the Hardy space $H^2$, and $\alpha = 1$ to the Dirichlet space $\mathcal{D}$. Note that the continuous inclusion $D_{\beta} \subsetneq D_{\gamma }$ holds for all $\gamma <\beta$, i.e.,
$ \Vert f\Vert_{\gamma} \leq \Vert f \Vert_{\beta}\textnormal{ for all } f\in D_{\beta}\textnormal{ and } \gamma <\beta.$
Moreover, when $\beta >1$ the spaces $D_{\beta} $ are continuously embedded in the disc algebra $\mathcal{A}$.

Dirichlet-type spaces are particular instances of general \textit{weighted Hardy spaces}, introduced by Shields \cite{Shields} to study weighted shifts in $\ell^2$. There is an extensive literature on these spaces, and we refer the reader to \cite[Chapter 2]{CMc}, for instance.\smallskip

Recall that an analytic function $u$ in $\mathbb{D}$ is a \emph{multiplier} of $D_{\alpha}$, if the analytic Toeplitz operator $T_u: f \mapsto u   f$ is defined everywhere on $D_{\alpha}$ (and hence bounded, by the Closed Graph Theorem). A well known fact about the  Dirichlet space is that the algebra $\MD$ of all the multipliers of $\D$ is not easy to describe. In particular, the strict inclusion $\MD \subset \D \cap H^\infty$ holds.
Indeed, the elements of $\MD$ were characterized by Stegenga \cite{St} in a notable paper, in terms of a condition involving the logarithmic capacity of their boundary values. We refer to \cite{Wu} for multipliers and Carleson measures in Dirichlet spaces and to \cite{EFKMR} for more on the subject of multipliers of $D_\alpha$.

In any case, it is not difficult to prove that every finite Blaschke product is a multiplier of $D_{\alpha}$  for all $\alpha \in \R$. Recall that a finite Blaschke product is given by
\[B(z)= e^{i\theta} \, \prod_{i=1}^{N} \frac{z-\alpha_i}{1-\overline{\alpha_i}z}, \qquad (z\in \mathbb{D})\]
where $\alpha_i\in \mathbb{D}$, counted according to its (prescribed) multiplicity. Finite Blaschke products play an important role in mathematics and connected areas such as complex geometry, linear algebra, operator theory and systems. We refer to the recent monograph \cite{GMR} for a detailed account of these results.

In order to analyze whether any finite Blaschke product $B$ satisfies the WSP in $D_{\alpha}$, we begin by considering the concrete example
$B(z)=z^2$ acting on the Bergman space $A^2$ (an open problem specifically posed in \cite{CaW}).

\subsection{Vector valued shifts}
The following approach is based on some ideas described in \cite{Nor, Par}. We may consider the space $A^2$ as a
direct (\emph{orthogonal}) sum of two copies of itself, $A^2_1$ and $A^2_2$, where a function $f \in A^2$ is decomposed as
\[f(z)= f_1(z^2)+zf_2(z^2).\] It is clear that $f \in A^2$ if and only if $f_1, f_2 \in A^2$, but we are imposing different equivalent
norms on each copy of $A^2$. We may think either $A_1^2$ and $A^2_2$ equipped with their usual norms and their sum $A^2$ equipped with the norm arising from such sum, or on the contrary, $A^2$ with usual norm decomposed as sum of two subspaces which inherit some comparable norm. We consider here the first of those choices. By doing so, we may view the operator $M_{z^2}$ as a diagonal matrix shift sending $(f_1,f_2) \in A^2 \oplus A^2$ to $(Sf_1, Sf_2)$. In this sense, $M_{z^2}$ may be expressed as
\begin{equation}\label{matshift}\begin{pmatrix} S & 0 \\ 0 & S
\end{pmatrix}.\end{equation}

The techniques developed by Nordgren when trying to solve Problem
151 in \cite{Hal} suggest a particular direction to study the
problem we have in mind. If $B(z)=z^k$ does not satisfy the WSP,
some closed invariant subspace $\mathcal{M}$ such that
\begin{equation}\label{eq1}
\mathcal{M}\neq [\mathcal{M}\ominus z^k \mathcal{M}]_{z^k}
\end{equation}
could, perhaps, be described through a
finite number of linear conditions. For instance, a finite number of
generators $h_1,...,h_r \in H^{\infty}$ multiplied by functions
$f_1,...,f_r \in A^2$, which, in addition, satisfy some finite number of
restrictions on their Taylor coefficients. It seems difficult to
come up with restrictions on the Taylor coefficients involving
coefficients of degree higher than $k$, and still generate a
non trivial closed invariant subspace of $A^2$. However, it
appears plausible that a counterexample  may be found for $M_{z^2}$ looking at how the
matrix operator \eqref{matshift} acts on the product space.

This is the idea behind the proofs of the following preliminary results, in which it is possible to guarantee that
a closed invariant subspace $\mathcal{M}$ for $M_{z^2}$, that is, $\mathcal{M}\in Lat(M_{z^2})$, is generated by $\mathcal{M}\ominus z^2\mathcal{M}$ whenever either $\mathcal{M}$ is also invariant for the shift, or decomposable as direct sum of
closed subspaces in each of the two copies of $A^2$, say $\mathcal{M}_1
\subset A^2_1$ and $\mathcal{M}_2 \subset A^2_2$.

\begin{prop}\label{prop1}
Let $\mathcal{M} \in Lat(M_z)$, then $\mathcal{M}=[ \mathcal{M} \ominus z^2 \mathcal{M} ]_{z^2}$.
\end{prop}

\begin{proof}
Since $\mathcal{M} \supset z \mathcal{M} \supset z^2 \mathcal{M}$, we have the decomposition \[  \mathcal{M}
\ominus z^2 \mathcal{M} = (\mathcal{M} \ominus z \mathcal{M}) \oplus  (z \mathcal{M} \ominus z^2 \mathcal{M}).\]

Since $\mathcal{M} \in Lat(M_z)$, Aleman, Richter and Sundberg's Theorem \cite{ARS} yields $\mathcal{M}= [\mathcal{M} \ominus z \mathcal{M}]_z$, or equivalently
$$
\mathcal{M}= \overline{\{pf: p \in \mathcal{P}, f
\in  \mathcal{M} \ominus z \mathcal{M}\}},$$
where $\mathcal{P}$ denotes the space of all polynomials.

We decompose $\mathcal{P}$ as the span of
$\widetilde{\mathcal{P}}_0:=\{p:p(z)=q(z^2), q \in \mathcal{P}\}$
and $\widetilde{\mathcal{P}}_1:=\{p:p(z)=zq(z^2), q \in
\mathcal{P}\}$. Choosing the induced norm in each copy of $A^2$, we
have that
\[\mathcal{M} \subset [ \mathcal{M} \ominus z\mathcal{M} ]_{z^2} \oplus [ z\mathcal{M} \ominus z^2 \mathcal{M}]_{z^2} \subset [ \mathcal{M} \ominus z^2 \mathcal{M} ]_{z^2}.\] The other inclusion
($[ \mathcal{M} \ominus z^2 \mathcal{M} ]_{z^2} \subset \mathcal{M}$) is always satisfied.
\end{proof}

\begin{prop}\label{prop2}
Let $\mathcal{M}=\mathcal{M}_1 \oplus \mathcal{M}_2$ with $\mathcal{M}_1 \subset A^2_1$, $\mathcal{M}_2 \subset
A^2_2$. Then $\mathcal{M}=[ \mathcal{M} \ominus z^2 \mathcal{M}]_{z^2}.$
\end{prop}

\begin{proof}
Since $\mathcal{M}_1$, $\mathcal{M}_2$ are shift invariant, it will be a direct
consequence of the main theorem in \cite{ARS}:
$$\mathcal{M}= \mathcal{M}_1 \oplus \mathcal{M}_2 = [ \mathcal{M}_1 \ominus z \mathcal{M}_1 ]_{z} \oplus
[ \mathcal{M}_2 \ominus z \mathcal{M}_2 ]_{z}.
$$

Since $\mathcal{M}\ominus z^2\mathcal{M}$ contains the
direct sums $\mathcal{M}_1 \ominus z\mathcal{M}_1$ and $\mathcal{M}_2 \ominus z\mathcal{M}_2$, then $\mathcal{M}$ is
generated by $\mathcal{M}\ominus z^2\mathcal{M}$. Again, the other inclusion ($[ \mathcal{M}
\ominus z^2 \mathcal{M} ]_{z^2} \subset \mathcal{M}$) is always satisfied.
\end{proof}

If we call $\mathcal{M}_1 = \mathcal{M} \cap A^2_1$ and $\mathcal{M}_2 = \mathcal{M} \cap A^2_2$, it is
necessarily true that $\mathcal{M}_1 \perp \mathcal{M}_2$, $\mathcal{M}_1$ and $\mathcal{M}_2$ are shift
invariant, and $\mathcal{M}_1 \oplus \mathcal{M}_2 \supset \mathcal{M}$ but $\mathcal{M}$ may be defined,
for instance, through restrictions between the $\mathcal{M}_1$ and $\mathcal{M}_2$
components.

On the other hand,  it is possible to provide invariant spaces $\mathcal{M}\in Lat(M_{z^2})$ not satisfying the
hypotheses of Propositions \ref{prop1} and \ref{prop2}, but such that $\mathcal{M}= [\mathcal{M} \ominus z^2 \mathcal{M}]_{z^2}$:

\begin{ex}\label{example 1}
Let $a \in \C \backslash \{0\}$, $h(z)=1+az$, and $\mathcal{M}=[ h ]_{z^2}$.
Then $\mathcal{M} = [\mathcal{M} \ominus z^2 \mathcal{M}]_{z^2}$ since $h$ is orthogonal to $z^2 \mathcal{M}$ and generates $\mathcal{M}$.
\end{ex}

Notice that in this case, if we denote $\mathcal{M}_i = \mathcal{M} \cap A_i^2$, we have
$\mathcal{M}_1=\mathcal{M}_2=A^2$ but $1 \in A^2 \backslash \mathcal{M}$, and $\mathcal{M} \in Lat(M_{z^2})
\backslash Lat(M_z)$. It can be shown that any space generated by a
finite collection of elements without any relations also provides similar examples (being the norm imposed on
$A^2$ the product of $k$ usual norms for collections of $k$
elements).

\subsection{Shimorin's Theorem} The main contribution regarding the wandering subspace property in a variety of spaces was carried out by Shimorin in \cite{Shimorin}. In particular, he showed that $\varphi(z)=z$ satisfies the WSP in $D_{\alpha}$ for $\alpha \in [0,1]$
since the operators of multiplication by $z$ are \emph{concave}, i.e., for every $x \in D_\alpha$, $\|T^2x\|^2-2\|Tx\|^2+\|x\|^2\leq 0$. For $\alpha \in [-1,0)$, the WSP follows as a consequence of the following result:
\begin{thm*}[Shimorin]
Let $T$ be a bounded operator in a Hilbert space $H$ such that the
following hold: \begin{itemize}
\item[(i)] $\bigcap_{n  \in \N} T^n H = \{0\}$
\item[(ii)] For $x$, $y \in H$, we have \[\|x+Ty\|^2 \leq 2(\|Tx\|^2 +
\|y\|^2).\]
\end{itemize}
Then $T$ has the wandering subspace property in $H$.
\end{thm*}

Observe that Shimorin's approach only applies to the \emph{usual norms} in
$D_\alpha$ (those described above). In the recent paper
 \cite{Secozk}, Seco has shown for each
$\alpha \in \R$ and each positive integer $k \geq 6$, the existence of an equivalent
norm $\|\cdot\|$ in $D_\alpha$ and $\mathcal{M}\in Lat(M_{z^k})$ that fails to have the wandering property with
respect to the norm $\|\cdot\|$, that is,
$$\mathcal{M}\neq [\mathcal{M}\ominus z^k \mathcal{M}]_{z^k} \mbox{ respect to } \|\cdot\|.$$
In particular, this is shown in some cases to be
the usual norm for $D_\alpha$: for instance when $k \geq 10$ and
$\alpha <-(5k+\frac{700}{(k-9)^2})$, or when $\alpha \in (-16-\varepsilon, -16+\varepsilon)$ and $k=6$, but numerical results
hint that for $\alpha <-4.2$ there might be $k$ large enough providing counterexamples (see also \cite{Nowak} for related results in this direction). The results in the next Section will establish, nevertheless, that by means of renormings it is possible to have the WSP for any finite Blaschke product.

\section{The wandering subspace property under renormings}\label{Sect 3}

In this section, we show that in any $D_{\alpha}$ with  $\alpha \in [-1,1]$ (where $\varphi(z)=z$ meets the WSP), given any finite Blaschke $B$ product, it is possible to renorm the space (with an equivalent norm) such that $B$ also has the WSP.

Before that, observe that Example \ref{example 1}  may be generalized to the case where instead of
$z^2$ we make use of any finite Blaschke product. For a function $f
\in H^2$, given any finite Blaschke product, $B$, it is clear from the Wold
decomposition that we can express $f$ as
\[f(z)= \sum_{k=0}^\infty B^k(z) h_k(z),\] where $h_k$ are functions
in the model space $H^2 \ominus BH^2$, and the norm of $f$ may be
found from those of $h_k$. Indeed, \[\|f\|_{0}^2 = \sum_{k=0}^\infty
\|h_k\|_0^2,\]
where, recall that $\|\cdot \|_0$ corresponds to the $H^2$-norm.

In \cite{CGGP}, the authors find an analogous expansion for the
$D_\alpha$ spaces:

\begin{thm}[Chalendar, Gallardo-Guti\'errez, Partington]\label{Thm CGP}
Let $\alpha \in [-1,1]$ and $B$ any finite Blaschke product. Then $f \in D_{\alpha}$ if and only if $f = \sum_{k=0}^\infty h_k B^k$ (convergence in $D_{\alpha}$ norm) with $h_k \in K_B$ and
$$\sum_{k=0}^\infty(k+1)^{\alpha}\|h_k\|_0^2 < \infty.$$
\end{thm}

\begin{rem}
The previous theorem was stated in \cite{CGGP} for $\alpha \in \{-1, 0, 1\}$ and $B(0)=0$, but
the same scheme of proof  works bearing in
mind two key facts:
\begin{itemize}
\item[(i)] Multiplication by any function in the model space $H^2\ominus BH^2$ is a
bounded operator.
\item[(ii)] Composition with the Blaschke product is a bounded
operator in $D_{\alpha}$.
\end{itemize}
These are both easy to check and the only parts of the proof that
generalize in a non-obvious way. The assumption $B(0)=0$ is not
really necessary since the spaces $B^n K_B$ are still mutually
orthogonal in $H^2$, and hence, linearly independent
finite-dimensional spaces.
\end{rem}

We are now in a position to state the following:

\begin{thm}\label{GGPS}
Let $\alpha \in [-1,1]$ and $B$ a finite Blaschke product. Then there exists a norm $\|\cdot\|_{B}$ under which $B$ has the wandering subspace property in $D_{\alpha}$, that is, for any $\mathcal{M} \in Lat(M_B)$ we have
\begin{equation*}
[\mathcal{M}\ominus B \mathcal{M}]_{B}=\mathcal{M} \quad \mbox{with respect to } \|\cdot\|_{B}.
\end{equation*}
Moreover, for $B(z)=z^k$ and $|\alpha| \in [0,\log(2)/\log(k+1)]$, the
norm $\|\cdot\|_{B}$ coincides with the usual $D_\alpha$ norm $\|\cdot \|_{\alpha}$.
\end{thm}

\begin{proof}
Given $B$ a finite Blaschke product, let $\|\cdot\|_{B}$ denote the norm defined by the
corresponding expression arising from Theorem \ref{Thm CGP}.
Then, the multiplication operator induced by $B$ acts exactly as the shift operator $M_z$ acts on
$D_\alpha$ with respect to $\|\cdot\|_{B}$;  therefore it satisfies property (ii) in Shimorin's
Theorem. Consequently, $M_B$ has the WSP.

The property (ii) for the shift in such spaces can be checked by
showing the following two properties:
\[\omega_1 \geq 1/2,\]
\[\omega_n (\omega_{n-1}+\omega_{n+1}) \leq 2 \omega_{n-1}
\omega_{n+1}, \quad n \geq 1,\]
where $\omega_n = (n+1)^\alpha$.

Finally, assume $B(z)=z^k$ and consider the usual norm $\|\cdot \|_{\alpha}$.  Let $\alpha \in
[-\frac{\log(2)}{\log(k+1)},0)$ and notice that in this case, the proof
of the second inequality above works in the same way as for $M_z$ with
$\omega_{n-1}$ substituted by $\omega_{n-k}$ and $\omega_{n+1}$
substituted by $\omega_{n+k}$. The first inequality is satisfied
substituting $\omega_1$ by $\omega_k$ precisely because $|\alpha|
\leq \frac{\log(2)}{\log(k+1)}$. If $\alpha \geq 0$ apply the same reasoning to $1/\omega_k$ to see that the operator is concave.\end{proof}

It seems worth mentioning that if we take $B(z)=z^2$ in Theorem \ref{GGPS}, the range of values of $\alpha$ for which the result holds without
renorming can actually be improved by moving the lower bound from $\alpha \geq -\log(2)/\log(3) \approx - 0.6309$ to $\alpha \geq -\log(2/3)/\log(5/3) \approx -0.7937$:

\begin{prop}\label{improve}
Let $\alpha \in [-\log(2/3)/\log(5/3), 0]$. Then the wandering subspace property holds for the operator of multiplication by $z^2$ in $D_\alpha$ equipped with its usual norm $\|\cdot \|_{\alpha}$.
\end{prop}

\begin{proof}
First note that we can define a norm in $D_\alpha$, given by a weight $\omega$ that makes multiplication by $z^2$ on $D_\alpha$
space satisfy the Shimorin condition just by changing the weights on
the first coordinate ($\omega_0 =\|1\|^2$): Indeed, define the weight $\omega$ by $\omega_k=(k+1)^\alpha$ for $k \geq 1$ and $\omega_0$ will be determined later. Condition (i) in Shimorin's Theorem is trivially satisfied and condition (ii) is equivalent to meeting all of the following:
\begin{itemize}
\item[(a)] $\omega_0 \leq 2 \omega_2.$
\item[(b)] $\omega_1 \leq 2 \omega_3.$
\item[(c)] $(1/\omega_{n-2} + 1/\omega_{n+2} -2/ \omega_{n}) \leq 0$ for all $n \geq 2.$
\end{itemize}

Property (b) is equivalent to $2^\alpha \leq 2^{2 \alpha +1}$, which is immediately checked since $\alpha \geq -1$. Standard calculus techniques show the validity of (c), for $n \geq 3$ and we are left with finding $\omega_0$ such that
\[ \frac{1}{2\cdot 3^{-\alpha} -  5^{-\alpha}} \leq  \omega_{0} \leq  2 \cdot 3^\alpha .\]
Therefore, if we assume
\[ 1 \leq (2\cdot 3^{-\alpha} -  5^{-\alpha})(2 \cdot 3^\alpha),\]
there is a valid choice of $\omega_0$ such that $\omega$ defines a norm in $D_\alpha$ for which the WSP holds. The latter equation is equivalent to
\[\alpha \geq \frac{\log(2/3)}{\log(5/3)}.\]
Now we know that for any $z^2$-invariant $\mathcal{M}$, the space $\mathcal{M} \ominus z^2 \mathcal{M}$ is exactly the same under the original norm and the new norm, and so even if the norm is different, whether or not the WSP holds does not change. So we get the desired result under the original norm.
\end{proof}

\begin{rem}
One could be inclined to think that the WSP for $z^k$ on $A^2$ follows from that on $z^k A^2$, shown in the previous proposition, based on its finite codimension as a subspace of $A^2$. However, it follows from \cite{Secozk} that $\C^{22} \oplus z^6 A^2$ fails to have the $z^6$ WSP if we equip $\C^{22}$ with the weight $\omega_t=(t+1)^{-16}$ for $t=0,...,21$. This space still contains $z^6 A^2$ as a finite codimension subspace.
\end{rem}

\begin{prop}\label{zkA2}
Let $k \in \N$, $\alpha \in [-1,0]$ and $\mathcal{M}=z^k D_\alpha$. Then $z^k$ has the wandering subspace property in $\mathcal{M}$.
\end{prop}

\begin{proof}
First, we can assume $\alpha <0$. For $s \in \N$ denote by $\omega_s = (s+1)^\alpha$. The condition (ii) of Shimorin's Theorem becomes equivalent to
\begin{itemize}
\item[(a)] $\omega_{s} \leq 2 \omega_{k+s}$, for all $s=k,...,2k-1$, and
\item[(b)] $(1/\omega_{s} + 1/\omega_{s+2k} -2/ \omega_{s+k}) \leq 0$ for all $s \geq k.$
\end{itemize}
To see (a), notice that the minimum of $(s+1)/(k+s+1)$ for $s=k,...,2k-1$ is achieved at $s=k$, that such minimum is therefore bigger than $1/2$ and that $\alpha \geq -1$.
In order to check (b), it suffices to see that the quantity \[g(s)= (s+1)^{-\alpha} - (s+k+1)^{-\alpha}\] is negative and increasing on $s$.  Negativity is clear since the exponent $-\alpha$ is positive and $(s+k+1) \geq (s+1)$. Moreover $g'(s) = |\alpha| ((s+1)^{-\alpha-1} - (s+k+1)^{-\alpha-1})$, which is positive since $\alpha \geq -1$.
\end{proof}

\begin{rem}
Proposition \ref{zkA2} may be interpreted as a property of the \emph{subspace} $z^k D_\alpha$ or as a property of the \emph{equivalent norm on $D_\alpha$} given by $\|f\|:=\|S^k f\|_{\alpha}$, that is, as a property of $D_\alpha$ with this particular choice of equivalent norm. In this sense, it yields a different proof of Theorem \ref{GGPS} for the case when $B$ is a monomial.
\end{rem}

We conclude with the following result, which reduces the question further.

\begin{cor}
Let $k \geq 1$, $\alpha \in [-1,0]$, and $\mathcal{M}$ be a $z^k$-invariant subspace of $D_{\alpha}$. Then \[\mathcal{M}=[\mathcal{M} \ominus z^{2k} \mathcal{M}]_{z^k}.\]
\end{cor}

\begin{proof}
Denote $T:=M_{z^k}$ acting on $D_\alpha$. Let $\mathcal{M}$ be a closed $T$-invariant subspace of $D_\alpha$. Then $\mathcal{N} := T\mathcal{M}$ is a $T$-invariant subspace. Moreover, $\mathcal{N} \subset TD_\alpha$ and hence, by Proposition \ref{zkA2} we have
\[\mathcal{N} = [\mathcal{N} \ominus T\mathcal{N}]_{T}.\]
Now we can see that \[\mathcal{M} \ominus T^2\mathcal{M} = (\mathcal{M} \ominus T\mathcal{M}) \oplus (T\mathcal{M} \ominus T^2\mathcal{M}).\]
So the smallest closed $T$-invariant subspace containing $\mathcal{M} \ominus T^2 \mathcal{M}$ contains both $\mathcal{M} \ominus \mathcal{N}$ and $\mathcal{N}$, and so, it is $\mathcal{M}$.
\end{proof}

\end{document}